\documentclass[11pt]{amsart}
\usepackage{latexsym,amscd,amssymb, graphicx, color, amsthm, bm}  
\usepackage{young}
\usepackage[margin=1in]{geometry}
\usepackage{hyperref}
\usepackage[all,cmtip]{xy}

\numberwithin{equation}{section}

\newtheorem{theorem}{Theorem}[section]

\newtheorem{lemma}[theorem]{Lemma}

\theoremstyle{definition}

\newcommand{\sign}{{\mathrm {sign}}}

\newcommand{\SYT}{{\mathrm {SYT}}}

\newcommand{\symm}{{\mathfrak{S}}}

\newcommand{\CC}{{\mathbb {C}}}
\newcommand{\BBB}{{\mathcal{B}}}
\newcommand{\CCC}{{\mathcal{C}}}
\newcommand{\WWW}{{\mathcal{W}}}

\begin{document}

\title[The polytabloid basis expands positively into the web basis]
{The polytabloid basis expands positively into the web basis}

\author{Brendon Rhoades}
\address
{Department of Mathematics \newline \indent
University of California, San Diego \newline \indent
La Jolla, CA, 92093-0112, USA}
\email{bprhoades@math.ucsd.edu}

\begin{abstract}
We show that the transition matrix from the polytabloid basis to the web basis of the irreducible
$\symm_{2n}$-representation of shape $(n,n)$ 
has nonnegative integer entries. This proves a conjecture of Russell and Tymoczko \cite{RT}.
\end{abstract}

\maketitle

\section{Main}
\label{Main}

The purpose of this note is to establish a positivity relation between the 
polytabloid and web 
bases of the Catalan-dimensional
 irreducible $\symm_{2n}$-representation of shape $(n,n)$. This proves a conjecture
 of Russell and Tymoczko \cite[Conj. 5.8]{RT}.

Let $V$ be a complex vector space with $\dim(V) = m < \infty$. Given two bases
$\BBB = \{ v_1, v_2, \dots, v_m \}$ and $\CCC = \{ w_1, w_2, \dots, w_m \}$ of $V$,
there exist unique $a_{ij} \in \CC$ such that 
\begin{equation}
v_i = \sum_{j = 1}^m a_{ij} w_j.
\end{equation}
Algebraic combinatorics gives many interesting examples of bases $\BBB$ and $\CCC$
where the transition matrix $A = (a_{ij})_{1 \leq i, j \leq m}$ has nonnegative real entries
and (with respect to an appropriate order on the bases) is {\em unitriangular}
(i.e. upper triangular with diagonal entries equal to 1).

Now let $V$ and $W$ be two complex vector spaces of the same dimension $m < \infty$.
Let $\BBB = \{v_1, v_2, \dots, v_m \}$ be a basis of $V$ and let 
$\CCC = \{w_1, w_2, \dots, w_m \}$ be a basis of $W$.  Assume further that a finite
group $G$ acts irreducibly on both $V$ and $W$, and that 
$V$ and $W$ are isomorphic $G$-modules. Let $\varphi: V \rightarrow W$
be a $G$-module isomorphism; by Schur's Lemma, $\varphi$ is unique up to 
a scalar.
The set $\varphi(\BBB) = \{ \varphi(v_1), \varphi(v_2), \dots, \varphi(v_m) \}$ is 
a basis of $W$, and we can again consider the transition matrix $A = (a_{ij})_{1 \leq i, j \leq m}$
defined by
\begin{equation}
\varphi(v_i) = \sum_{j = 1}^m a_{ij} w_j.
\end{equation}
The matrix $A$ is uniquely determined up to a nonzero scalar, and we can again
ask whether its entries are nonnegative real numbers (up to an appropriate scaling)
and whether $A$ is unitriangular (with respect to an appropriate order on the bases, after
scaling).

Recall that the irreducible modules $S^{\lambda}$ for the symmetric group $\symm_m$
are naturally labeled by partitions $\lambda \vdash m$.
We consider the case where $V = V_n$ and $W = W_n$ are two different models for the 
irreducible representation $S^{(n,n)}$ of the symmetric group $\symm_{2n}$ on $2n$ letters
corresponding to the partition $(n,n) \vdash 2n$.  
The dimension of $V_n$ and $W_n$ is given by the $n^{th}$ Catalan number
$\mathrm{Cat}(n) = \frac{1}{n+1} {2n \choose n}$.
The space $V_n$ is the module $S^{(n,n)}$ 
taken with respect to the {\em polytabloid basis} (otherwise known as 
{\em Young's natural representation}) 
and the space $W_n$ is the module 
$S^{(n,n)}$ taken with respect to the basis of {\em webs} (attached 
to the Lie algebra $\mathfrak{sl}_2$).

The $\symm_{2n}$-module $V_n$ is defined as follows.
Recall that an {\em $(n,n)$-tableau} is a filling of the 
$2 \times n$ Ferrers shape $(n,n)$ with the numbers $1, 2, \dots, 2n$.
If $T$ an $(n,n)$-tableau, the {\em (row) tabloid} $\{T\}$ is the set of all $(n!)^2$
tableaux obtainable from $T$ by permuting its entries within rows.  
Let $M^{(n,n)}$ be the $\CC$-vector space formally spanned by these tabloids:
\begin{equation}
M^{(n,n)} := \mathrm{span}_{\CC} \{ \{T \} \,:\, \text{$T$ an $(n,n)$-tableau} \}.
\end{equation}
The symmetric group $\symm_{2n}$ acts on $(n,n)$-tableaux by letter permutation;
this induces an action of $\symm_{2n}$ on $M^{(n,n)}$ given by
$\sigma. \{T\} := \{ \sigma.T \}$.  The module $M^{(n,n)}$ is isomorphic to the induction
of the trivial representation from $\symm_n \times \symm_n$ to $\symm_{2n}$.

Given any $(n,n)$-tableau $T$, let $C_T \subseteq \symm_{2n}$ be the 
subgroup of permutations in $\symm_{2n}$ which stabilize the columns of $T$. 
The {\em Young antisymmetrizer} is the group algebra element
\begin{equation}
\varepsilon_T := \sum_{\sigma \in C_T} \sign(\sigma) \cdot \sigma \in \CC[\symm_{2n}]
\end{equation}
and the associated {\em polytabloid} is
\begin{equation}
v_T := \varepsilon_T .\{T \} \in M^{(n,n)}.
\end{equation}
For any $\sigma \in \symm_{2n}$, we have $\sigma.v_T = v_{\sigma(T)}$, so  the linear
space 
\begin{equation}
V_n := \mathrm{span}_{\CC} \{ v_T \,:\, \text{$T$ an $(n,n)$-tableau} \} \subseteq
M^{(n,n)}
\end{equation}
spanned by these polytabloids carries the structure of a $\symm_{2n}$-module.
It turns out that $V_n$ is irreducible of shape $(n,n)$.

Recall that an $(n,n)$-tableau $T$ is {\em standard} if its entries increase going down
columns and across rows. Let $\SYT(n,n)$ denote the set of all standard $(n,n)$-tableaux;
we have $|\SYT(n,n)| = \mathrm{Cat}(n)$.   The five tableaux of $\SYT(3,3)$ are shown below.
\begin{equation*}
\begin{Young}
1 & 3 & 5 \cr
2 & 4 & 6 
\end{Young} \hspace{0.2in}
\begin{Young}
1 & 3 & 4 \cr
2 & 5 & 6 
\end{Young} \hspace{0.2in}
\begin{Young}
1 & 2 & 5 \cr
2 & 3 & 6 
\end{Young} \hspace{0.2in}
\begin{Young}
1 & 2 & 4 \cr
3 & 5 & 6 
\end{Young} \hspace{0.2in}
\begin{Young}
1 & 2 & 3 \cr
4 & 5 & 6 
\end{Young}
\end{equation*}
We let $T_0 \in \SYT(n,n)$ be the
tableau whose first row entries are $1, 3, 5, \dots$; when
$n = 3$ this is the leftmost tableau shown above.

It can be shown that the set $\{ v_T \,:\, T \in \SYT(n,n) \}$ of all standard polytabloids
forms a basis of $V_n$.
This basis, first discovered by Young \cite{Y} and studied by Specht \cite{S}, will be called the
{\em polytabloid basis}. 
\footnote{Russell and Tymoczko refer to this basis as the {\em Specht basis}. We avoid
this terminology to prevent confusion with other `Specht bases' in symmetric
group representation theory (see e.g. \cite{ATY}).}
Its definition extends readily to any partition $\lambda$,
yielding a basis for the correspding irreducible symmetric group module $S^{\lambda}$.
It is perhaps the most well-known basis of $S^{\lambda}$.

Whereas the module $V_n$ is defined in terms of tableaux, the module $W_n$
is defined in terms of matchings.  A {\em web} of order $n$ (attached to the Lie algebra
$\mathfrak{sl}_2$) is a perfect matching $M$ on the set $[2n] := \{1, 2, \dots, 2n \}$ which is
{\em noncrossing}: for any $a < b < c < d$, we do not simultaneously have
$a \sim c$ and $b \sim d$ in $M$. 
Let $\WWW_n$ be the set of all order $n$ webs.  The five webs of $\WWW_3$ are shown 
below.
\begin{center}
\includegraphics[scale = 0.17]{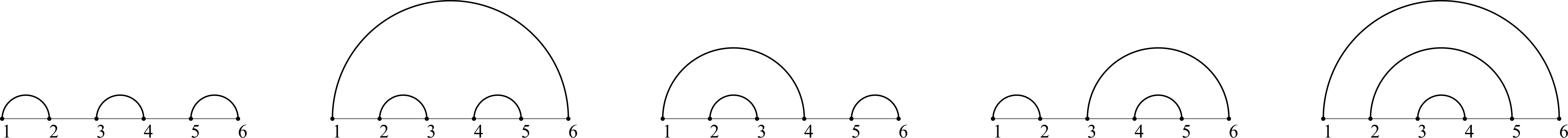}
\end{center}

Let
\begin{equation}
W_n := \mathrm{span}_{\CC} \{ w_M \,:\, \text{$M \in \WWW_n$} \}
\end{equation}
be the vector space with basis $\WWW_n$; we have
$\dim(W_n) = \mathrm{Cat}(n)$.
We let $M_0 \in \WWW_n$ 
be the web given by $2i-1 \sim 2i$ for all $1 \leq i \leq n$; when $n = 3$ this is the leftmost 
web shown above.

The symmetric group $\symm_{2n}$ acts on $W_n$ as follows.  Given $1 \leq i \leq 2n-1$,
 let $s_i = (i,i+1) \in \symm_{2n}$ be the corresponding adjacent transposition.
 If $M$ is an order $n$ web, the action of $s_i$ on $w_M$ is given by
 \begin{equation}
 \label{s-action-on-w}
 s_i.w_M = \begin{cases}
 - w_M & \text{if $i \sim i+1$ in $M$,} \\
 w_M + w_{M'} & \text{otherwise,}
 \end{cases}
 \end{equation}
 where in the second branch $M'$ is the web obtained from $M$ by replacing the pairs
 $a \sim i$ and $b \sim i+1$ with the pairs $a \sim b$ and $i \sim i+1$.
This action respects the Coxeter relations of $\symm_{2n}$, and 
 so endows $W_n$ with the structure of an $\symm_{2n}$-module (see \cite{PPR}).
 The module $W_n$ is irreducible of shape $(n,n)$ (see \cite{PPR}); the basis
 $\{ w_M \,:\, M \in \WWW_n \}$ is the {\em web basis} of $W_n$.
 The web basis is related to the invariant theory of the Lie group $\mathrm{SL}_2$.
 
 Since $V_n$ and $W_n$ are isomorphic and irreducible $\symm_{2n}$-modules,
 there exists an isomorphism $\varphi: V_n \rightarrow W_n$ which is uniquely determined
 up to a scalar. Russell and Tymoczko used combinatorial techniques to show that, after scaling,
 the map
 $\varphi$ sends $v_{T_0}$ to $w_{M_0}$ and relates the polytabloid basis to the web basis 
 in a unitriangular way.
 
 \begin{theorem} (Russell-Tymoczko \cite[Cor. 5.4, Thm. 5.5]{RT})
 \label{unitriangular}
 There exists a unique $\symm_{2n}$-module isomorphism 
 $\varphi: V_n \rightarrow W_n$ satisfying $\varphi(v_{T_0}) = w_{M_0}$. Furthermore,
 the transition matrix $(a_{TM})$ 
 between the polytabloid and web bases given by
 \begin{equation}
 \varphi(v_T) = \sum_{M \in \WWW_n} a_{TM} w_M \quad (T \in \SYT(n,n)),
 \end{equation}
 is unitriangular with respect to an appropriate order on the bases.
 \end{theorem}
 
 Russell and Tymoczko conjectured 
 \cite[Conj. 5.8]{RT}
 that the entries $a_{TM}$ are nonnegative; we prove their
 conjecture here.
 
 \begin{theorem}
 \label{positive}
 The entries $a_{TM}$ of the transition matrix of Theorem~\ref{unitriangular}
 are nonnegative integers.
 \end{theorem}

Given any partition $\lambda \vdash m$, the irreducible representation $S^{\lambda}$
of the symmetric group $\symm_m$ has a number of interesting bases, suggesting
possible generalizations of Theorems~\ref{unitriangular} and \ref{positive}.
The polytabloid basis $\{ v_T \,:\, T \in \SYT(\lambda) \}$ 
can be constructed as above, and so may be compared with various extensions of the 
web basis.
\begin{itemize}
\item  For any $\lambda \vdash m$, the {\em Kazhdan-Lusztig cellular basis} \cite{KL} is 
a basis of the $\symm_m$-irreducible $S^{\lambda}$.
It can be shown that the web and KL cellular bases coincide when $\lambda = (n,n)$
and $m = 2n$.
Garsia and McLarnan proved that the transition matrix between the 
polytabloid and KL
cellular bases is unitriangular, but that its entries can be negative for general
shapes $\lambda$ \cite{GM}.
Theorem~\ref{positive} shows that this transition matrix {\em does} have nonnegative
entries when $\lambda$ is a $2 \times n$ rectangle; it is an open problem to determine
the partitions $\lambda$ for which this matrix has nonnegative entries.
\item When $\lambda = (k, k, 1^{m-2k})$ is a {\em flag-shaped partition},
the {\em skein basis} of $S^{\lambda}$ has entries labeled by noncrossing set partitions
of $\{1, 2, \dots, m \}$ with $k$ total blocks and no singleton blocks.
This basis was introduced in \cite{R} to give algebraic proofs of cyclic sieving results
and coincides with the web and KL cellular bases when $m = 2k$ (but differs from the KL cellular
basis in general).  Again, the transition basis between the polytabloid and skein bases 
can have negative entries for general flag-shaped partitions $\lambda$.
\item  When $\lambda = (n,n,n)$ is a $3 \times n$ rectangle, 
there is a basis of $S^{(n,n,n)}$ indexed by order $n$ webs attached to the Lie algebra
$\mathfrak{sl}_3$.  These are graphs embedded in a disk with $3n$ boundary points
satisfying certain conditions; see \cite{PPR} for details on these webs, and how 
$\symm_{3n}$ acts on the vector space spanned by them.
This web basis is related to the invariant theory of $\mathrm{SL}_3$.
Computational evidence suggests that the transition matrix between the polytabloid and 
web bases of $S^{(n,n,n)}$ is unitriangular with nonnegative entries.
\end{itemize}

Our strategy for proving Theorem~\ref{positive} is to use an alternative model for $W_n$
in terms of products of matrix minors. The lack of such a 
simple model for the $\mathfrak{sl}_3$-web
basis of $S^{(n,n,n)}$ is the primary obstruction to extending our method from the two-row
to the three-row case.

\section{Proof of Theorem~\ref{positive}}
\label{Proof}

The main idea is to recast the vector space $W_n$ using the work of Kung and Rota
on binary forms \cite{KR}.  To this end, let $x = (x_{ij})_{1 \leq i \leq 2, 1 \leq j \leq 2n}$
be a $2 \times 2n$ matrix of variables. We work in the polynomial ring 
$\CC[x] := \CC[x_{ij} \,:\, 1 \leq i \leq 2, 1 \leq j \leq 2n]$ generated by these variables.
For any two-element subset 
$\{a < b\} \subseteq [2n]$, let 
\begin{equation}
\label{syzygy}
\Delta_{ab} := x_{1a} x_{2b} - x_{1b} x_{2a}
\end{equation}
be the maximal minor of $x$ with column set $\{a,b\}$.
These minors satisfy the following syzygy:
for $a < b < c < d$ we have
\begin{equation}
\Delta_{ac} \Delta_{bd} = \Delta_{ab} \Delta_{cd} + \Delta_{ad} \Delta_{bc}.
\end{equation}

If $M = \{I_1, I_2, \dots, I_n \}$ is any perfect matching on $[2n]$
(noncrossing or otherwise),
so that $I_1, I_2, \dots, I_n$ are $2$-element subsets of the column set of $x$,
we set
\begin{equation}
\Delta_M := \Delta_{I_1} \Delta_{I_2} \cdots \Delta_{I_n}.
\end{equation}

\begin{lemma}
\label{positive-expansion}
For any perfect matching $M$ on $[2n]$, there are nonnegative integers
$c_{MN}$ such that
$\Delta_M = \sum_{N \in \WWW_n} c_{MN} \Delta_{N}$.
\end{lemma}

\begin{proof}
Recall that $\WWW_n$ is the set of noncrossing perfect matchings on 
$[2n]$.  If the matching $M$ does not have any crossings, we are done.
Otherwise, there exist $a < b < c < d$ such that $a \sim c$ and $b \sim d$ in $M$;
this quadruple of indices is said to form a {\em crossing pair}.
Applying the syzygy relation \eqref{syzygy} gives
\begin{equation}
\Delta_M = \Delta_{M'} + \Delta_{M''},
\end{equation}
where the perfect matchings $M'$ and $M''$ are identical to $M$ except that 
($a \sim b$ and $c \sim d$) in $M'$ and ($a \sim d$ and $b \sim c$) in $M''$.
Since both $M'$ and $M''$ have strictly fewer total crossing pairs
than $M$, we are done by induction.
\end{proof}

The symmetric group $\symm_{2n}$ acts on the matrix $x$ of variables by column
permutation. For any $\sigma \in \symm_{2n}$ and any perfect matching
$M = \{I_1, I_2, \dots, I_n\}$, we have
\begin{equation}
\label{sigma-delta-action}
\sigma.\Delta_M = \pm \Delta_{\sigma(M)},
\end{equation}
where $\sigma(M) = \{\sigma(I_1), \sigma(I_2), \dots, \sigma(I_n)\}$; the sign
is determined by the number of pairs in $I_1, I_2, \dots, I_n$ which are inversions
of the permutation $\sigma$.
Equation~\eqref{sigma-delta-action} shows that the $\CC$-vector subspace
\begin{equation}
\label{module-one}
\mathrm{span}_{\CC} \{ \Delta_M \,:\, 
\text{$M$ a perfect matching on $[2n]$} \}
\end{equation}
of $\CC[x]$
carries the structure of a $\symm_{2n}$-module.

This module \eqref{module-one}
 is  $W_n$ in disguise. To see this, we apply Lemma~\ref{positive-expansion}
to get the equality of subspaces of $\CC[x]$:
\begin{equation}
\mathrm{span}_{\CC} \{ \Delta_M \,:\, 
\text{$M$ a perfect matching on $[2n]$} \} =
\mathrm{span}_{\CC} \{ \Delta_M \,:\, 
\text{$M \in \WWW_n$} \}. 
\end{equation}
If $M \in \WWW_n$ is a noncrossing perfect
matching on $[2n]$ and $1 \leq i \leq 2n-1$, 
the syzygy \eqref{syzygy} shows that
\begin{equation}
\label{s-action-on-delta}
s_i.\Delta_M = \begin{cases}
- \Delta_M & \text{if $i \sim i+1$ in $M$,} \\
\Delta_M + \Delta_{M'} & \text{otherwise,}
\end{cases}
\end{equation}
where in the second branch $M'$ is obtained from $M$ by replacing the pairs
$a \sim i$ and $b \sim i+1$ with the pairs $a \sim b$ and $i \sim i+1$.
Comparing Equations~\eqref{s-action-on-w}
and \eqref{s-action-on-delta}, we see that the linear map 
\begin{equation}
\psi: W_n \rightarrow 
\mathrm{span}_{\CC} \{ \Delta_M \,:\, \text{$M \in \WWW_n$} \}
\end{equation}
defined by $\psi(w_M) = \Delta_M$ for all $M \in \WWW_n$ is a surjective map of 
$\symm_{2n}$-modules. 
Since $W_n$ irreducible, the map $\psi$ is an isomorphism
which sends the web basis $\{ w_M \,:\, M \in \WWW_n \}$ to the 
minor product basis $\{ \Delta_M \,:\, M \in \WWW_n\}$.  We therefore identify 
$W_n
= \mathrm{span}_{\CC} \{ \Delta_M \,:\, \text{$M \in \WWW_n$} \}$ 
and 
$w_M = \Delta_M$.

We are ready to prove Theorem~\ref{positive}.  Let $\varphi: V_n \rightarrow W_n$ 
be the $\symm_{2n}$-module isomorphism of Theorem~\ref{unitriangular}
satisfying 
\begin{equation}
\label{first}
\varphi(v_{T_0}) = w_{M_0} = \Delta_{M_0}.
\end{equation}
Let $T \in \SYT(n,n)$ be an arbitrary standard tableau and let $\sigma \in \symm_{2n}$
be the permutation defined by $\sigma(T_0) = T$.  Then $\sigma(M_0)$ is a perfect
matching on $\{1, 2, \dots, 2n\}$ (which is not necessarily noncrossing).  Applying 
$\sigma$ to Equation~\eqref{first} gives
\begin{equation}
\label{second}
\varphi(v_T) = \varphi(v_{\sigma(T_0)}) =
\varphi(\sigma.v_{T_0}) = \sigma.\varphi(v_{T_0}) =
\sigma.\Delta_{M_0} = \pm \Delta_{\sigma(M_0)},
\end{equation}
where the second equality used the action of $\symm_{2n}$ on polytabloids
and the sign could {\em a priori} depend on $T$.

By Lemma~\ref{positive-expansion}, the polynomial
$\Delta_{\sigma(M_0)}$ appearing in Equation~\eqref{second}
expands nonnegatively and with integer coefficients into the web basis
$\{ \Delta_M \,:\, M \in \WWW_n \}$. 
Theorem~\ref{positive} will  follow
if we can show that the sign appearing in Equation~\eqref{second} is always positive.
But this follows from the unitriangularity result
Theorem~\ref{unitriangular} of Russell and Tymoczko \cite{RT} combined
with Lemma~\ref{positive-expansion}.

\section{Acknowledgments}
\label{Acknowledgments}

The author is grateful to Heather Russell and Julianna Tymoczko for helpful 
conversations. The author was partially supported by NSF Grant DMS-1500838.


\begin{thebibliography}{99}
 
 \bibitem{ATY}  S. Ariki, T. Terasoma, and H.-F. Yamada.
 Higher Specht polynomials.
 {\it Hiroshima Math. J.},
 {\bf 27 (1)} (1997), 177--188.
 
 \bibitem{GM}  A. M. Garsia and T. J. McLarnan.
 Relations between Young's natural and the Kazhdan-Lusztig representations
 of $S_n$.
 {\it Adv. Math.}, {\bf 69 (1)} (1988), 32--92.
 
 \bibitem{KL}
 D. Kazhdan and G. Lusztig.
 Representations of Coxeter groups and Hecke algebras.
 {\it Invent. Math.}, {\bf 53} (1979), 165--184.
 
 \bibitem{KR} J. P. S. Kung and G.-C. Rota. The invariant theory of binary forms.
 {\it Bull. Amer. Math. Soc.}, {\bf 10 (1)} (1984), 27--85.
 
 \bibitem{PPR} K. Petersen, P. Pylyavskyy, and B. Rhoades.
 Promotion and cyclic sieving via webs.
 {\it J. Algebraic Combin.}, {\bf 30 (1)} (2009), 19--41.
 
 \bibitem{R} B. Rhoades. A skein action of the symmetric group on noncrossing partitions.
 {\it J. Algebraic Combin.}, {\bf 45 (1)} (2017), 81--127.
 
\bibitem{RT}  H. Russell and J. Tymoczko.  The transition matrix between the Specht and web bases
is unipotent with additional vanishing entries. To appear,
{\it Int. Math. Res. Not.}, 2017. {\tt arXiv:1701.01868}.


\bibitem{S} W. Specht. Die irreduziblen Darstellungen der symmetrischen Gruppe.
{\it Math. Z.}, {\bf 39 (1)} (1935), 696--711.

\bibitem{Y} A. Young.
The Collected Papers of Alfred Young.
{\it Math. Expositions}, No. 21, Univ. of Toronto. Toronto Press (1873-1940).

  
\end{thebibliography}
\end{document}